\documentclass[12pt,oneside,reqno]{amsart}

\usepackage{amsmath,amssymb,amsthm,textcomp}
\usepackage{amsfonts,graphicx}
\usepackage[mathscr]{eucal}
\usepackage{color}
\usepackage{float}
\usepackage{diagbox}
\usepackage{csquotes}
\usepackage[backend=bibtex,%
firstinits=true,%
doi=false,%
isbn=true,%
url=false,%
maxnames=99]{biblatex}%

\AtEveryBibitem{\clearfield{issn}}
\AtEveryCitekey{\clearfield{issn}}
\numberwithin{equation}{section}
\DeclareNameAlias{sortname}{last-first}
\theoremstyle{definition}
\usepackage{mathtools}
\numberwithin{equation}{section}

\newcommand{\ncom}{\newcommand}

\ncom{\beq}{\begin{equation}}
	\ncom{\eeq}{\end{equation}}
\ncom{\bea}{\begin{eqnarray*}}
	\ncom{\eea}{\end{eqnarray*}}
\ncom{\beqa}{\begin{eqnarray}}
	\ncom{\eeqa}{\end{eqnarray}}
\ncom{\nno}{\nonumber}
\ncom{\non}{\nonumber}
\ncom{\ds}{\displaystyle}
\ncom{\half}{\frac{1}{2}}
\ncom{\mbx}{\makebox{.25cm}}
\ncom{\hs}{\mbox{\hspace{.25cm}}}
\ncom{\rar}{\rightarrow}
\ncom{\Rar}{\Rightarrow}
\ncom{\noin}{\noindent}
\ncom{\bc}{\begin{center}}
	\ncom{\ec}{\end{center}}
\ncom{\sz}{\scriptsize}
\ncom{\rf}{\ref}
\ncom{\s}{\sqrt{2}}
\ncom{\sgm}{\sigma}
\ncom{\Sgm}{\Sigma}
\ncom{\psgm}{\sigma^{\prime}}
\ncom{\dt}{\delta}
\ncom{\Dt}{\Delta}
\ncom{\lmd}{\lambda}
\ncom{\Lmd}{\Lambda}
\ncom{\Th}{\Theta}
\ncom{\e}{\eta}
\ncom{\eps}{\varepsilon}
\ncom{\pcc}{\stackrel{P}{>}}
\ncom{\lp}{\stackrel{L_{p}}{>}}
\ncom{\dist}{{\rm\,dist}}
\ncom{\sspan}{{\rm\,span}}
\ncom{\re}{{\rm Re\,}}
\ncom{\im}{{\rm Im\,}}
\ncom{\sgn}{{\rm sgn\,}}
\ncom{\ba}{\begin{array}}
	\ncom{\ea}{\end{array}}
\ncom{\hone}{\mbox{\hspace{1em}}}
\ncom{\htwo}{\mbox{\hspace{2em}}}
\ncom{\hthree}{\mbox{\hspace{3em}}}
\ncom{\hfour}{\mbox{\hspace{4em}}}
\ncom{\vone}{\vskip 2ex}
\ncom{\vtwo}{\vskip 4ex}
\ncom{\vonee}{\vskip 1.5ex}
\ncom{\vthree}{\vskip 6ex}
\ncom{\vfour}{\vspace*{8ex}}
\ncom{\norm}{\|\;\;\|}
\ncom{\integ}[4]{\int_{#1}^{#2}\,{#3}\,d{#4}}
\ncom{\vspan}[1]{{{\rm\,span}\{ #1 \}}}
\ncom{\dm}[1]{ {\displaystyle{#1} } }
\ncom{\ri}[1]{{#1} \index{#1}}

\newtheorem{theorem}{\bf Theorem}[section]
\newtheorem{remark}{\bf Remark}[section]

\newtheorem{lemma}{Lemma}[section]

\newtheoremstyle
{remarkstyle}
{}
{11pt}
{}
{}
{\bfseries}
{:}
{     }
{\thmname{#1} \thmnumber{#2} }

\theoremstyle{remarkstyle}

\def\eps{\varepsilon}

\date{\today}
\begin{document}
\title{Amnesic Elephant Random Walk with Polynomially Decaying Step Sizes}
\author[Shyan Ghosh]{Shyan Ghosh}
\address{Shyan Ghosh, Department of Mathematics, Indian Institute of Technology Bhilai, Durg 491002, India.}
\email{shyanghosh@iitbhilai.ac.in}
\author[Manisha Dhillon]{Manisha Dhillon}
\address{Manisha Dhillon, Department of Mathematics, Indian Institute of Technology Bhilai, Durg 491002, India.}
\email{manishadh@iitbhilai.ac.in}
\author[Kuldeep Kumar Kataria]{Kuldeep Kumar Kataria}
\address{Kuldeep Kumar Kataria, Department of Mathematics, Indian Institute of Technology Bhilai, Durg 491002, India.}
\email{kuldeepk@iitbhilai.ac.in}

\begin{abstract}
In this paper, we introduce an amnesic elephant random walk with polynomially decaying step sizes. The walk retains the loss of memory effect of amnesic ERW, however its step sizes decay polynomially over time. We study the effect of step-size exponent and memory parameter on the long-time behaviour of the walk. Two critical thresholds are identified that determine its phase diagram. We obtain almost sure convergence results, law of iterated logarithm, asymptotic normality and mean square displacement rate of the walk across different parameter regimes. The polynomial decay of step sizes gives rise to a subdiffusive regime while classical diffusion is recovered in the absence of decay. Also, we identify localization regimes in which the walk converges almost surely to a finite random variable.
\end{abstract}
\maketitle

\section{Introduction}
The elephant random walk (ERW) is a one dimensional integer valued random walk with a complete memory of its past. It was introduced by Sch\"utz and Trimper (2004) to study the effect of long-range memory on simple random walk. In the standard ERW, the entire history is treated uniformly, that is, every previous time instant has equal probability of being selected. More precisely, at each time point, the walker selects one of its past increments uniformly at random and either repeats or reverses it according to a fixed probability $p\in[0,1]$. The memory parameter $p$ quantifies the tendency of walker to repeat a previously taken step. For $p=1/2$, the effect of memory vanishes, and the resulting process becomes Markovian. 

Over the past two decades, the ERW and its extensions have attracted considerable attention. In Baur and Bertoin (2016), the authors established a connection between the ERW and P\'olya-type urn models, and based on this connection obtained a functional central limit theorem. In Bercu (2018), a martingale approach is used to characterize the long-term behaviour of the walk in different regimes that includes the law of large numbers, law of the iterated logarithm (LIL), quadratic strong law and asymptotic normality results. It is shown that the limiting distribution in superdiffusive regime is non-Gaussian. Gaussian fluctuations around this limiting distribution were subsequently established in  Gu\'erin \textit{et al.} (2026), while further properties and detailed descriptions of the limiting distribution in the superdiffusive regime are discussed in  Gu\'erin \textit{et al.} (2025). Some other interesting models such as the ERW with restricted or partial memory is studied in Gut and Stadtm\"uller (2021a, 2021b), ERW with stops is considered by Bercu (2022), and by using the scaling limits for Markov chains and self-similar Markov processes, the number of zeros of ERW is studied by Bertoin (2022). Also, the ERW with random step sizes is studied by Dedecker \textit{et al.} (2023) and Fan and Shao (2024). The standard ERW in multidimensional setting is discussed by Bercu and Laulin (2019) and Bercu (2025). Other aspects of the walk, for example, the center of mass is studied by Bercu and Laulin (2021), recurrence and transience are discussed by Curien and Laulin (2024) and Qin (2025). Its various other generalizations are considered by Gangopadhyay and Maulik (2022), Maulik \textit{et al.} (2025), Roy \textit{et al.} (2025), Peres and Qin (2026) and Dhillon and Kataria (2026), \textit{etc.} Laulin (2022) introduced an extension of the ERW by incorporating a smooth amnesia mechanism. Related ideas involving memory loss have also been explored in reinforced random walks, for example, Bertenghi and Laulin (2025) considered a step-reinforced random walk with amnesic memory. Recently, Nakano (2025) introduced and studied the ERW with polynomially decaying step sizes where several limiting results that provide a complete description of the long-term behaviour of the walk are obtained.

In this paper, we study the polynomial decay of step sizes for the amnesic elephant random walk model by Laulin (2022). That is, the walk retains amnesic memory mechanism while the increment at time instant $n$ is scaled by $n^{-\gamma}$, $\gamma\ge0$. We call it the amnesic elephant random walk with polynomially decaying step sizes. Two critical thresholds  $C_\beta=(4\beta+3)/(4\beta+4)$ and $C_{\beta,\gamma}=(2\beta+\gamma+1)/(2\beta+2)$ govern its phase diagram. Unlike the amnesic ERW, the polynomial decay of step sizes gives rise to a subdiffusive regime between $C_\beta$ and $C_{\beta,\gamma}$. Also, we show that the walk exhibits localization in certain parameter regimes in the sense that it converges almost surely to a finite random variable.

The paper is organized as follows:
In Section \ref{prel}, we set some notations and collect some known results related to the  amnesic ERW of Laulin (2022). In Section \ref{interm}, we introduce the amnesic ERW with polynomially decaying step sizes, and construct a suitable martingale that serves as the main tool for deriving its asymptotic results. Here, two critical  thresholds $C_\beta$ and $C_{\beta,\gamma}$ are observed. The lemmas given in this section are used later to study the asymptotic behaviour of the walk. In Section \ref{mainsec}, we obtain the asymptotic results for the walk across different parameter regimes that includes the almost sure convergence result, LIL, asymptotic normality and mean square displacement rate. Our main aim is to study the influence of step-size exponent and memory parameter on the long-time behaviour of the walk. Also, it is shown that in certain regimes the walk is localized.

\section{Preliminaries}\label{prel}
Here, we collect some known results related to the amnesic ERW $\{W_n\}_{n\ge 1}$ of Laulin (2022). First, we set some notations that will be used throughout the paper. Let $\mathbb{I}_G$ denote the indicator function of a set $G$. For two sequences $\{x_n\}_{n\ge 1}$ and $\{y_n\}_{n\ge 1}$, the notation $x_n \sim y_n$ stands for $x_n/y_n \to 1$ as $n \to \infty$, and if these sequences are positive then  $x_n=O(y_n)$ indicates that $\{x_n/y_n\}_{n\ge 1}$ is bounded. For any matrix $A$, $A^t$ denotes its transpose, $\lambda_{\max}(A)$ be its maximum eigenvalue and $\|A\|=\sqrt{\lambda_{\max}(A^tA)}$. Also, $\xrightarrow{p}$ and $\xrightarrow{d}$ denotes convergence in probability and convergence in distribution, respectively. In expressions involving conditional expectations, the abbreviation a.s. for almost surely is omitted to avoid repetition.

\subsection{Amnesic ERW}\label{amnesicerw}
 In this model, the first increment $X_1$ has Rademacher distribution $\mathcal{R}(q)$, $0\le q\le 1$. For $n\ge 2$, the walker chooses past steps according to $\beta(n)$ which has the following distribution (see Laulin (2022)):
\begin{equation*}\label{betadef}
\mathbb{P}\{\beta(n)=k\}=\frac{(\beta+1)\Gamma(n)\Gamma(k+\beta)}{\Gamma(n+\beta+1)\Gamma(k)}\\
=\frac{\beta+1}{n}\frac{\mu_k}{\mu_{n+1}},\ \  k=1,2, \dots, n,
\end{equation*}
where 
\begin{equation*}
\mu_n=\prod_{k=1}^{n-1}\Big(1+\frac{\beta}{k}\Big)=\frac{\Gamma(n+\beta)}{\Gamma(n)\Gamma(\beta+1)}.
\end{equation*}
The position of the walker at time instant $n+1$ is given by
$W_{n+1}=X_1+X_2+\dots+X_{n+1}$. Here, $X_{n+1}= \alpha_{n+1}X_{\beta(n)}$ is the $(n+1)$-th increment of the walk, where $\alpha_{n+1}$ has Rademacher distribution $\mathcal{R}(p)$, $0\le p\le 1$. It is assumed that $\alpha_{n+1}$ and $\beta(n)$ are independent, and both are independent of $\{X_1,X_2,\dots, X_n\}$.
	
Let $\mathcal{F}_n=\sigma(X_1,X_2,\dots, X_n)$. Then,
\begin{equation}\label{lauamninc}
\mathbb{E}(X_{n+1}|\mathcal{F}_n)=\frac{a(\beta+1)}{n\mu_{n+1}}Y_n,
\end{equation}
where $a=(2p-1)$ and $Y_n=\mu_1X_1+\mu_2X_2+\dots+\mu_nX_n$. So,
\begin{equation*}
\mathbb{E}(Y_{n+1}|\mathcal{F}_n)=\Big(1+\frac{a(\beta+1)}{n}\Big)Y_n,\ n\ge 1.
\end{equation*}
For $n\ge 1$, let 
$a_n=\prod_{k=1}^{n-1}(1+a(\beta+1)/k)^{-1}.
$
Then,  
\begin{equation}\label{asyma_namn}
\lim_{n \to \infty} n^{a(\beta+1)} a_n
	= \Gamma\bigl(a(\beta+1)+1\bigr)
\end{equation}
and
\begin{equation}\label{asymmu_namn}
\lim_{n \to \infty} n^{-\beta}\mu_n
	= \frac{1}{\Gamma(\beta+1)}.
\end{equation}
Also, let 
\begin{equation}\label{Mnpdec}
	M_n\coloneqq a_nY_n.
\end{equation} 
Then, $\{M_n\}_{n\ge1}$ is  $\{\mathcal{F}_n\}_{n\ge 1}$-martingale.

Laulin (2022) observed that the asympotic behaviour of martingale $\{M_n\}_{n\ge1}$ is determined by that of $s_n=a_1^2\mu_1^2+a_2^2\mu_2^2+\dots+a_n^2\mu_n^2$. 
The following three different regimes corresponding to the behaviour of $s_n$ are obtained: the diffusive regime for $0\le p<C_\beta$, the critical regime for $p=C_\beta$ and the superdiffusive regime for $p>C_\beta$, where $C_\beta=(4\beta+3)/4(\beta+1)$. 
\begin{lemma}
The following limiting results hold true:\\
\noindent (i) For $0\le p<C_\beta$, we have
\begin{equation}\label{eqsndiff}
\lim_{n\to\infty}\frac{s_n}{n^{2\beta-2a(\beta+1)+1}}=\frac{1}{2\beta-2a(\beta+1)+1}\Big(\frac{\Gamma(a(\beta+1)+1)}{\Gamma(\beta+1)}\Big)^2.
\end{equation}           
\noindent (ii) For $ p=C_\beta$, we have
	\begin{equation}\label{sncrit}
	\lim_{n\to\infty}\frac{s_n}{\log n}=\Big(\frac{\Gamma\!\left(\beta+3/2\right)}{\Gamma(\beta+1)}\Big)^2.
\end{equation}
\noindent (iii) For $C_\beta<p\le1$, we have
\begin{equation*}
\lim_{n\to\infty}s_n=\sum_{k=1}^{\infty}
\Big(\frac{\Gamma(a(\beta+1)+1)\Gamma(k+\beta)}{\Gamma(k+a(\beta+1))\Gamma(\beta+1)}\Big)^2<\infty.
\end{equation*}
\end{lemma}

\begin{lemma}\label{MnO}
Let $\{M_n\}_{n\ge 1}$ be as defined in \eqref{Mnpdec}. Then, for any $\epsilon>0$, the following results hold:\\
\noindent (i) if $0\le p<C_\beta$ then $M_n^2=O((\log n)^{1+\epsilon} n^{2\beta-2a(\beta+1)+1})$ a.s.,\vspace{.1cm}\\
\noindent (ii) if $p=C_\beta$ then $		M_n^2=O((\log\log n)^{1+\epsilon}\log n)$  a.s., and\vspace{.1cm}\\
\noindent (iii) if $C_\beta<p\le 1$ then $		\lim_{n \to \infty}M_n=M$ a.s., 
	where $M$ is some finite random variable.
\end{lemma}
\begin{lemma}\label{LemMnLIL}
The following LIL hold true for $\{M_n\}_{n\ge 1}$:\\
\noindent (i) For $0\le p<C_\beta$, we have
\begin{align*}
\limsup_{n \to \infty}\frac{M_n}{\sqrt{2n^{2\beta-2a(\beta+1)+1}\log\log n}}&=-\liminf_{n\to\infty}\frac{M_n}{\sqrt{2n^{2\beta-2a(\beta+1)+1}\log\log n}}\\
&=\frac{|\Gamma(a(\beta+1)+1)|}{\Gamma(\beta+1)\sqrt{2\beta-2a(\beta+1)+1}}\ \text{a.s.}
\end{align*}
\noindent (ii) For $p=C_\beta$, we have
\begin{align*}
\limsup_{n \to \infty}\frac{M_n}{\sqrt{2(\log n)\log\log\log n}}&=-\liminf_{n \to \infty}\frac{M_n}{\sqrt{2(\log n)\log\log\log n}}=\frac{\Gamma(\beta+3/2)}{\Gamma(\beta+1)}\ \text{a.s.}
\end{align*}
\end{lemma}

On substituting $d=1$ in Lemma A.14 of Chen and Laulin (2023), we get
\begin{equation}\label{EYnsq} 
\mathbb{E}(Y_n^2)\sim \frac{n^{2a(\beta+1)}}{\Gamma(2a(\beta+1)+1)}+\frac{n^{2\beta+1}}{(\beta+1)(2\beta-2a(\beta+1)+1)(\Gamma(\beta+1))^2}\ \text{as $n\to\infty$}.
\end{equation}

\section{Amnesic ERW with polynomially decaying step sizes}\label{interm}
In this section, we introduce and study an amnesic elephant random walk with polynomially decaying step sizes. From Section \ref{amnesicerw}, recall that $X_n$, $n\ge 1$ are the increments of amnesic ERW. For $\gamma\ge 0$, let $T_k=X_k/k^\gamma$ for all $k=1,2,\dots,n$. We define the amnesic ERW with polynomially decaying step sizes $\{S_n\}_{n\geq 1}$ as follows: $S_n\coloneqq T_1+T_2+\dots+T_n$. 

For $\gamma=0$, the walk $\{S_n\}_{n\geq 1}$ reduces to the amnesic ERW of Laulin (2022).

Let  $\mathcal{F}_n=\sigma(X_1,X_2,\dots, X_n)$ and $M_n$ be as defined in \eqref{Mnpdec}. Then, from \eqref{lauamninc}, we have
\begin{equation*}
\mathbb{E}(T_{n+1}|\mathcal{F}_n)=\frac{1}{(n+1)^{\gamma}}\frac{a(\beta+1)}{n\mu_{n+1}}Y_n.
\end{equation*}
Therefore,
\begin{equation*}
\mathbb{E}(S_{n+1}|\mathcal{F}_n)=S_n+\frac{1}{(n+1)^{\gamma}}\frac{a(\beta+1)}{n\mu_{n+1}}Y_n.
\end{equation*}

Also, let 
\begin{equation}\label{Nnpdec}
N_n\coloneqq S_n-ac_nM_n,
\end{equation}
where $c_1=1$ and
\begin{equation}\label{cndelamns}
c_{n}=\sum_{k=1}^{n-1}\frac{\beta+1}{k(k+1)^\gamma a_k \mu_{k+1}},\  n\ge 2.
\end{equation}
Then, $\{N_n\}_{n\ge1}$ is  $\{\mathcal{F}_n\}_{n\ge 1}$-martingale.

We shall use the following notation throughout the paper: $C_{\beta,\gamma}=(2\beta+\gamma+1)/2(\beta+1)$.

The next result is obtained by using \eqref{asyma_namn} and \eqref{asymmu_namn} in \eqref{cndelamns}. 
\begin{lemma}\label{cnasymlem}
The asymptotic behaviour of $\{c_n\}_{n\ge 1}$ as $n\to\infty$ is given as follows:\\
\noindent (i) For $0\le p<C_{\beta,\gamma}$, we have $\lim_{n\to\infty}	c_n=c_\infty$, 
where $c_\infty$ is some non-zero constant.\\
\noindent (ii) For $p=C_{\beta,\gamma}$, we have
\begin{equation}\label{cnasymcritt}
c_n\sim \frac{\Gamma(\beta+2)}{\Gamma(\beta+\gamma+1)}\log n.
\end{equation}
\noindent (iii) For $C_{\beta,\gamma}<p\le 1$, we have
\begin{equation}\label{cnasymsupdiffu}
c_n\sim \frac{\Gamma{(\beta+2)}n^{a(\beta+1)-\beta-\gamma}}{(a(\beta+1)-\beta-\gamma)\Gamma{(a(\beta+1)+1})}.
\end{equation}
\end{lemma}
	
Our aim is to obtain the limiting results for $\{S_n\}_{n\ge 1}$. We proceed as follows:

Let us consider the following vector martingale:
\begin{equation}\label{twotemmar}
\mathcal{M}_n=\begin{pmatrix}
		M_n\vspace{.2cm}\\
		N_n
\end{pmatrix},
\end{equation}
where $\{M_n\}_{n\ge 1}$ and $\{N_n\}_{n\ge 1}$ are locally square-integrable $\{\mathcal{F}_n\}_{n\ge1}$-martingales given in \eqref{Mnpdec} and \eqref{Nnpdec}, respectively.
Also, let 
$\mathcal{M}_0=\begin{pmatrix}
		0 &
		0
\end{pmatrix}^t$,
that is, $M_0=N_0=0$. Then, we have
\begin{equation}\label{deltavecMn}
\Delta \mathcal{M}_{n}=\mathcal{M}_{n}-\mathcal{M}_{n-1} =\begin{pmatrix}
			a_{n}\mu_{n}\\
			b_{n}
		\end{pmatrix}\delta_{n},
\end{equation}
where $a=2p-1$, $\delta_1=X_1$, $\delta_{n+1}=X_{n+1}-a(\beta+1)Y_n/(n\mu_{n+1})$ and
$b_{n}=n^{-\gamma}-a a_nc_{n}\mu_{n}$ for all $n\ge 1$. So, from \eqref{lauamninc}, we get 
\begin{equation}\label{deltsnboun}
\mathbb{E}(\delta_{n+1}|\mathcal{F}_n)=0,\ \ \  \sup_{n\ge 1} \mathbb{E}(\delta_{n+1}^2|
\mathcal{F}_n)\le 1\ \ \text{and}\ \  \sup_{n\ge 1} \mathbb{E}(|\delta_{n+1}|^3|
\mathcal{F}_n)<\infty.
\end{equation}
From \eqref{deltavecMn}, we obtain the predictable quadratic variation $\langle \mathcal{M}\rangle_n$ of the vector martingale $\{\mathcal{M}_n\}_{n\ge 1}$ as follows:
\begin{align}
\langle \mathcal{M}\rangle_n&=\sum_{k=1}^{n}\mathbb{E}((\Delta \mathcal{M}_k)(\Delta \mathcal{M}_k)^t|\mathcal{F}_{k-1})\nonumber\\
&=\sum_{k=1}^{n}Q_k-(a(\beta+1))^2\sum_{k=1}^{n-1}\frac{Q_{k+1}Y_k^2}{(k\mu_{k+1})^2},\label{angMncal}
\end{align}
where 
\begin{equation*}
Q_k=\begin{pmatrix}
	a_k^2\mu_k^2 & a_k\mu_kb_k\vspace{.1cm}\\
	a_k\mu_kb_k & b_k^2
\end{pmatrix}.
\end{equation*}
Let $u_n=b_1^2+b_2^2+\dots+b_n^2$. Then, we have
\begin{align}
\langle M\rangle_n&\coloneqq\sum_{k=1}^{n}\mathbb{E}((\Delta M_k)^2|\mathcal{F}_{k-1})=s_n-(a(\beta+1))^2\sum_{k=1}^{n-1}\Big(\frac{a_{k+1}Y_k}{k}\Big)^2,\label{Mnpredqu}\\
\langle N\rangle_n&\coloneqq\sum_{k=1}^{n}\mathbb{E}((\Delta N_k)^2|\mathcal{F}_{k-1})=u_n-(a(\beta+1))^2\sum_{k=1}^{n-1}\Big(\frac{b_{k+1}Y_k}{k\mu_{k+1}}\Big)^2\label{Nnpredqu}
\end{align}
and
\begin{equation}\label{MnNnpredqu}
\langle M,N\rangle_n
\coloneqq\sum_{k=1}^{n}\mathbb{E}((\Delta M_k)(\Delta N_k)|\mathcal{F}_{k-1})=\sum_{k=1}^{n}a_k\mu_kb_k-(a(\beta+1))^2\sum_{k=1}^{n-1}\frac{a_{k+1}b_{k+1}Y_k^2}{k^2\mu_{k+1}}.
\end{equation}
The proof of next two lemmas follow from \eqref{asyma_namn}, \eqref{asymmu_namn} and Lemma \ref{cnasymlem}.
\begin{lemma}\label{lembnasym}
As $n\to\infty$, the following asymptotic results hold true:\\
\noindent (i) For $0\le p<C_{\beta,\gamma}$, we have
\begin{equation}\label{bndiffu}
b_n\sim \begin{cases}
\frac{-ac_\infty \Gamma(a(\beta+1)+1)}{\Gamma(\beta+1)}n^{\beta-a(\beta+1)}\ \text{if}\ p\ne\frac{1}{2},\vspace{.1cm}\\
n^{-\gamma}\ \text{if}\ p=\frac{1}{2}.
\end{cases}
\end{equation}
\noindent (ii) For $p=C_{\beta,\gamma}$, we have
\begin{equation}\label{bncrit}
b_n\sim \begin{cases}
-(\beta+\gamma)n^{-\gamma}\log n\ \text{if}\ p\ne\frac{1}{2},\vspace{.1cm}\\
n^{-\gamma}\ \text{if}\ p=\frac{1}{2}.
\end{cases}
\end{equation}
\noindent (iii) For $C_{\beta,\gamma}<p\le 1$ and $(\beta,\gamma)\ne (0,0)$, we have
\begin{equation}\label{bnsupdiffu}
b_n\sim \frac{\beta+\gamma}{\beta+\gamma-a(\beta+1)} n^{-\gamma}.
\end{equation}
\end{lemma}
\begin{lemma}\label{unasymlem}
As $n\to\infty$, the following asymptotic results hold true: \\
\noindent (i) For $0\le\gamma<1/2$, we have
\begin{equation*}
u_n\sim\begin{cases}
	\Big(\frac{ac_\infty \Gamma(a(\beta+1)+1)}{\Gamma(\beta+1)}\Big)^2\frac{n^{2\beta-2a(\beta+1)+1}}{2\beta-2a(\beta+1)+1}\ \text{if}\ 0\le p<C_{\beta,\gamma},\ p\ne \frac{1}{2},\vspace{.1cm}\\
		\frac{n^{1-2\gamma}}{1-2\gamma}\ \text{if}\ p=\frac{1}{2}<C_{\beta,\gamma},\vspace{.1cm}\\
			\frac{(\beta+\gamma)^2}{1-2\gamma}n^{1-2\gamma}(\log n)^2\  \text{if}\ p=C_{\beta,\gamma},\vspace{.1cm}\\
			\Big(\frac{\beta+\gamma}{\beta+\gamma-a(\beta+1)}\Big)^2 \frac{n^{1-2\gamma}}{1-2\gamma}\  \text{if}\ C_{\beta,\gamma}<p\le 1,
		\end{cases}
\end{equation*}
where $c_\infty$ is as given in Lemma \ref{cnasymlem}.\\
	\noindent (ii) For $\gamma=1/2$, we have
	\begin{equation*}
		u_n\sim\begin{cases}
			\Big(\frac{ac_\infty \Gamma(a(\beta+1)+1)}{\Gamma(\beta+1)}\Big)^2\frac{n^{2\beta-2a(\beta+1)+1}}{2\beta-2a(\beta+1)+1}\ \text{if}\ 0\le p<C_{\beta,\gamma},\ p\ne \frac{1}{2},\vspace{.1cm}\\
			\log n\ \text{if}\ p=\frac{1}{2}<C_{\beta,\gamma},\vspace{.1cm}\\
			\frac{(\beta+\gamma)^2}{3}(\log n)^3\ \text{if}\ p=C_{\beta,\gamma},\vspace{.1cm}\\
			\Big(\frac{\beta+\gamma}{\beta+\gamma-a(\beta+1)}\Big)^2 \log n\  \text{if}\ C_{\beta,\gamma}<p\le 1.
		\end{cases}
	\end{equation*}
	\noindent (iii) For $1/2<\gamma\le 1$, we have
	\begin{equation*}
		u_n\sim\begin{cases}
			\Big(\frac{ac_\infty \Gamma(a(\beta+1)+1)}{\Gamma(\beta+1)}\Big)^2\frac{n^{2\beta-2a(\beta+1)+1}}{2\beta-2a(\beta+1)+1}\ \text{if}\ 0\le p<C_{\beta},\, p\ne 1/2,\vspace{.1cm}\\
			l_1\ \text{if}\ 0\le p<C_{\beta},\, p= 1/2,\vspace{.1cm}\\
			\Big(\frac{ac_\infty\Gamma(\beta+3/2)}{\Gamma(\beta+1)}\Big)^2\log n\ \text{if}\ p=C_{\beta},\vspace{.1cm}\\
			l_2\ \text{if}\ C_{\beta}<p<C_{\beta,\gamma},\vspace{.1cm}\\
			l_3\ \text{if}\ p=C_{\beta,\gamma}\ne 1,\vspace{.1cm}\\
			l_4\ \text{if}\ p=C_{\beta,\gamma}= 1,\vspace{.1cm}\\
			l_5\ \text{if}\ C_{\beta,\gamma}<p\le 1, \ (\beta,\gamma)\ne (0,0),
		\end{cases}
	\end{equation*}
where $l_1=\sum_{n=1}^{\infty}n^{-2\gamma}$, $l_2=((ac_\infty \Gamma(a(\beta+1)+1))^2/(\Gamma(\beta+1))^2)\sum_{n=1}^{\infty}n^{2\beta-2a(\beta+1)}$, 
$l_3=(\beta+\gamma)^2\sum_{n=1}^{\infty}(n^{-\gamma}\log n)^2$,
$l_4=(\beta+1)^2\sum_{n=1}^{\infty}(n^{-1}\log n)^2$ and 
$l_5=((\beta+\gamma)^2/(\beta+\gamma-a(\beta+1))^2)\sum_{n=1}^{\infty}n^{-2\gamma}$
are non-zero constants. 
\end{lemma}

From \eqref{twotemmar} and \eqref{deltavecMn}, we have $N_n=b_1\delta_1+b_1\delta_2+\dots+b_n\delta_n$. In view of \eqref{deltsnboun}, the proof of Lemma \ref{lemNnLLN} follows from Theorem 1.3.15 and Theorem 1.3.24 of Duflo (1997). Also, the proof of Lemma \ref{lemNnLIL} follows from Corollary 6.4.25 of Duflo (1997). 
\begin{lemma}\label{lemNnLLN}
Let $\{N_n\}_{n\ge 1}$ be the martingale  defined in \eqref{Nnpdec}. Then, we have \\
\noindent (i) $N_n^2=O(u_n\log u_n)$ a.s. whenever $\lim_{n \to \infty}u_n=\infty$,  and\vspace{.1cm}\\
\noindent (ii) $\lim_{n \to \infty}N_n=N$ a.s. whenever $\lim_{n \to \infty}u_n<\infty$. Here, $N$ is some finite random variable.
\end{lemma}
\begin{lemma}\label{lemNnLIL}
Let $\lim_{n \to \infty}u_n=\infty$ and
 $\sum_{n=1}^{\infty}b_n^3u_n^{-3/2}<\infty$. Then, the following LIL holds:
\begin{equation*}
\limsup_{n \to \infty}\frac{|N_n|}{\sqrt{2u_n\log\log u_n}}\le 1\ \text{a.s.}
\end{equation*}
\end{lemma}
From Lemma \ref{lemNnLLN}, observe that the asymptotic behaviour of $\{N_n\}_{n\ge 1}$ is determined by that of $\{u_n\}_{n\ge 1}$. Thus, from Lemma \ref{unasymlem}, the asymptotic behaviour of $\{N_n\}_{n\ge 1}$ exhibits phase transitions at critical thresholds $C_\beta$ and $C_{\beta,\gamma}$.

It is important to note that \eqref{EYnsq} does not hold for $2a(\beta+1)=2\beta+1$, that is, $p=C_\beta$. In this case, the following asymptotic behaviour of $\mathbb{E}(Y_n^2)$ follows from the second equation on p. 38 of Chen and Laulin (2023). More precisely,
\begin{equation}\label{Ynsqas2}
\mathbb{E}(Y_n^2)
\sim
\frac{n^{2a(\beta+1)}}{\Gamma(2a(\beta+1)+1)}
+
\frac{n^{2a(\beta+1)}}{(\beta+1)(\Gamma(\beta+1))^2}
\sum_{k=1}^{n-1}\frac{1}{k},
\end{equation}
where we have used $2a(\beta+1)=2\beta+1$. As
the second term on right side of \eqref{Ynsqas2} dominates, we have the following result:
\begin{lemma}
If $p=C_\beta$ then
\begin{equation}\label{critYn}
\mathbb{E}(Y_n^2)\sim\frac{n^{2a(\beta+1)}\log n}
{(\beta+1)\bigl(\Gamma(\beta+1)\bigr)^2}\ \text{as $n\to\infty$}.
\end{equation}
\end{lemma}
\section{Asymptotic Results for the Walk}\label{mainsec}
Here, we discuss the asymptotic results related to the amnesic ERW with polynomially decaying step sizes $\{S_n\}_{n\ge 1}$ for different regimes arising due to the critical thresholds $C_\beta$ and $C_{\beta,\gamma}$. These are broadly categorized in the following three cases:
 Case I.  $0\le\gamma<1/2$, that is, $C_{\beta,\gamma}<C_\beta$, 
Case II.  $\gamma=1/2$, that is, $C_{\beta,\gamma}=C_\beta$ and Case III.  $\gamma>1/2$, that is, $C_\beta< C_{\beta,\gamma}$.
The classification of these regimes is summarized in Table 1.
\begin{table}[H]
	\centering
	\label{tabRegimes}	
		\renewcommand{\arraystretch}{1.5}
	\setlength{\tabcolsep}{10pt}
	
	\begin{tabular}{|c|c|c|c|}
		\hline
		
		\diagbox[width=2.7cm,height=1.5cm]{Subcase}{Case}
		& I 
		& II
		& III \\
		\hline
		
		a
		& $0\le p<C_{\beta,\gamma}$
		& $0\le p<C_\beta$
		& $0\le p<C_\beta$ \\
		\hline
		
		b
		& $p=C_{\beta,\gamma}$
		& $p=C_\beta$
		& $p=C_\beta$ \\
		\hline
		
		c
		& $C_{\beta,\gamma}<p<C_\beta$
		& $C_\beta<p\le1$
		& $C_\beta<p<C_{\beta,\gamma}$ \\
		\hline
		
		d
		& $p=C_\beta$
		& --- 
		& $p=C_{\beta,\gamma}<1$ \\
		\hline
		
		e
		& $C_\beta<p\le1$
		& --- 
		& $C_{\beta,\gamma}<p\le1$ \\
		\hline
		
		f
		& --- & --- & $p=C_{\beta,\gamma}=1$ \\
		\hline
	\end{tabular}\vspace{.2cm}
\centering \caption{Regimes classification by relative ordering of $C_\beta$ and $C_{\beta,\gamma}$.}
\end{table}

\subsection{Case I} In this case, 
we discuss the almost sure convergence, LIL, asymptotic normality and mean square displacement rate for $\{S_n\}_{n\ge1}$ across different regimes.
\subsubsection*{Case I(a)}  Here, $0\le p<C_{\beta,\gamma}<C_\beta$.
\begin{theorem} 
Let $0\le p<C_{\beta,\gamma}$. Then,\\
\noindent (i) for $p\ne 1/2$, we have
\begin{equation*}
\lim_{n \to \infty}\frac{S_n}{n^{2\beta-2a(\beta+1)+1}}=0\ \text{a.s.},
\end{equation*}
\noindent (ii) for $p=1/2$, we have
\begin{equation*}
\lim_{n \to \infty}\frac{S_n}{n^{1-2\gamma}}=0\ \text{a.s.}
\end{equation*}
\end{theorem}
\begin{proof}
As $C_{\beta,\gamma}<C_\beta$, for any $\epsilon>0$, from Lemma \ref{MnO}(i), we have $M_n^2=O((\log n)^{1+\epsilon}$ $ n^{2\beta-2a(\beta+1)+1})\ \text{a.s.}$
Thus, we have
$(c_nM_n)^2=O(n^{2\beta-2a(\beta+1)+1}(\log n)^{1+\epsilon})\ \text{a.s.}$
which follows from Lemma \ref{cnasymlem}(i). Therefore, 
\begin{equation}\label{cnMnasymi}
\lim_{n \to \infty}\frac{c_nM_n}{n^{2\beta-2a(\beta+1)+1}}=0\ \text{a.s.}
\end{equation}
Now, from Lemma \ref{unasymlem}(i) and Lemma \ref{lemNnLLN}, we get 
\begin{equation*}
N_n^2=\begin{cases}
	O(n^{2\beta-2a(\beta+1)+1}\log n)\ \text{a.s. if}\ p\ne 1/2,\\
	O(n^{1-2\gamma}\log n)\ \text{a.s. if}\ p=1/2.
\end{cases}
\end{equation*}
Thus, 
\begin{align}\label{Nnasymi}
\left.\begin{aligned}
\lim_{n \to \infty}\frac{N_n}{n^{2\beta-2a(\beta+1)+1}}&=0\ \text{a.s. if}\ p\ne1/2,\\
\lim_{n \to \infty}\frac{N_n}{n^{1-2\gamma}}&=0\ \text{a.s. if}\ p=1/2.\\
\end{aligned}
\right\}
\end{align}
Now, the required results follow on using \eqref{cnMnasymi} and \eqref{Nnasymi} in \eqref{Nnpdec}.
\end{proof}
Next, we obtain a result related to LIL for $\{S_n\}_{n\ge 1}$.
\begin{theorem}\label{thmste}
Let $0\le p<C_{\beta,\gamma}$. Then,\\
\noindent (i) for $p\ne 1/2$, we have
\begin{equation}\label{LILsupp}
\limsup_{n\to\infty}\frac{|S_n|}{\sqrt{2n^{2\beta-2a(\beta+1)+1}\log\log n}}\le \frac{2|ac_\infty \Gamma(a(\beta+1)+1)|}{\Gamma(\beta+1)\sqrt{2\beta-2a(\beta+1)+1}}\ \text{a.s.},
\end{equation}
\noindent (ii) for $p= 1/2$, we have 
\begin{equation}\label{LILeqsupp}
\limsup_{n\to\infty}\frac{S_n}{\sqrt{2n^{1-2\gamma}\log \log n}}=-\liminf_{n\to\infty}\frac{S_n}{\sqrt{2n^{1-2\gamma}\log \log n}}\nonumber\\
=\frac{1}{\sqrt{1-2\gamma}}\ \text{a.s.}
\end{equation}
\end{theorem}
\begin{proof}
From Lemma \ref{LemMnLIL}(i) and Lemma \ref{cnasymlem}(i), we get
\begin{align}
\limsup_{n\to\infty}\frac{ac_nM_n}{\sqrt{2n^{2\beta-2a(\beta+1)+1}\log \log n}}&=-\liminf_{n\to\infty}\frac{ac_nM_n}{\sqrt{2n^{2\beta-2a(\beta+1)+1}\log \log n}}\nonumber\\
&=\frac{|ac_\infty\Gamma(a(\beta+1)+1)|}{\Gamma(\beta+1)\sqrt{2\beta-2a(\beta+1)+1}}\ \text{a.s.}\label{reg1Mn}
\end{align}
Also, from Lemma \ref{lembnasym}(i), Lemma \ref{unasymlem}(i) and Lemma \ref{lemNnLIL}, we obtain
\begin{equation}\label{reg1Nn1}
\limsup_{n\to\infty}\frac{|N_n|}{\sqrt{2n^{2\beta-2a(\beta+1)+1}\log\log n}}\le \frac{|ac_\infty\Gamma(a(\beta+1)+1)|}{\Gamma(\beta+1)\sqrt{2\beta-2a(\beta+1)+1}}\ \text{a.s.}
\end{equation}
So, by using \eqref{reg1Mn} and \eqref{reg1Nn1} in \eqref{Nnpdec}, we get \eqref{LILsupp}.

For $p=1/2$, from \eqref{Nnpdec}, it follows that  $S_n=N_n$. Let $\tilde{N}_n=N_n-\mathbb{E}(N_1)$. Since $\delta_{n+1}=X_{n+1}-a(\beta+1)Y_n/(n\mu_{n+1})$ and $b_{n}=n^{-\gamma}-a a_nc_{n}\mu_{n}$, we have $|\Delta \tilde{N}_n|\le K_nK_n^{-1}$, where $K_n=\sqrt{(2\log\log\langle\tilde{N}\rangle_n)/\langle\tilde{N}\rangle_n}$. Thus, from Lemma \ref{unasymlem}(i), we get $\lim_{n \to \infty}K_n=0$ a.s. Now, from Theorem 1 and Theorem 2 of Stout (1970), we obtain
\begin{equation}\label{limsupg}
\limsup_{n\to\infty}\frac{S_n}{\sqrt{2n^{1-2\gamma}\log \log n}}=\frac{1}{\sqrt{1-2\gamma}}\ \text{a.s.}
\end{equation}
Similarly, on considering $\tilde{N}_n=-N_n+\mathbb{E}(N_1)$, we have
\begin{equation}\label{liminfg}
-\liminf_{n\to\infty}\frac{S_n}{\sqrt{2n^{1-2\gamma}\log \log n}}=\frac{1}{\sqrt{1-2\gamma}}\ \text{a.s.}
\end{equation}
Finally, on combining \eqref{limsupg} and \eqref{liminfg}, we get \eqref{LILeqsupp}. This completes the proof.
\end{proof}
\begin{theorem}\label{asymnn}
Let $p=1/2<C_{\beta,\gamma}$. Then,
\begin{equation*}
\frac{S_n}{\sqrt{n^{1-2\gamma}}}\xrightarrow{d}\mathcal{N}\Big(0,\frac{1}{1-2\gamma}\Big).
\end{equation*}
\end{theorem}
\begin{proof}
By using Lemma \ref{unasymlem}(i) in \eqref{Nnpredqu}, we get
\begin{equation*}
\lim_{n \to \infty}\frac{\langle N\rangle_n}{n^{1-2\gamma}}=\frac{1}{1-2\gamma}\ \text{a.s.}
\end{equation*}
Also, for any $\epsilon>0$, we have
\begin{equation*}
\frac{1}{n^{1-2\gamma}}\sum_{k=1}^{n}\mathbb{E}((\Delta N_k)^2\mathbb{I}_{\{|\Delta N_k|\ge \epsilon\sqrt{n^{1-2\gamma}}\}}|\mathcal{F}_{k-1})\le \frac{1}{\epsilon^2n^{2-4\gamma}}\sum_{k=1}^{n}k^{-4\gamma},
\end{equation*}
and as $0\le \gamma<1/2$, we have
$
\lim_{n \to \infty}n^{2\gamma-1}\sum_{k=1}^{n}\mathbb{E}((\Delta N_k)^2\mathbb{I}_{\{|\Delta N_k|\ge \epsilon\sqrt{n^{1-2\gamma}}\}}|\mathcal{F}_{k-1})=0\ \text{a.s.}
$
The required result follows from Corollary 2.1.10 of Duflo (1997).
\end{proof}
\subsubsection*{Case I(b)}
Here, $p=C_{\beta,\gamma}<C_\beta$ and $(\beta,\gamma)\ne (0,0)$.
\begin{theorem}
Let $p=C_{\beta,\gamma}$. Then,  $\lim_{n \to \infty}S_n/n^{1-2\gamma}=0$ a.s.
\end{theorem}
\begin{proof}
From \eqref{cnasymcritt} and Lemma \ref{MnO}(i), for any $\epsilon>0$, we get $(c_nM_n)^2=O(n^{1-2\gamma}$ $(\log n)^{3+\epsilon})\ \text{a.s.}$
 Thus, 
\begin{equation}\label{Lim2}
\lim_{n\to\infty}\frac{c_nM_n}{n^{1-2\gamma}}=0\ \text{a.s.}
\end{equation}
Also, from Lemma \ref{unasymlem}(i) and Lemma \ref{lemNnLLN}(i), we have
$N_n^2=O(n^{1-2\gamma}(\log n)^{3})$ a.s. So, 
\begin{equation}\label{Lim1}
\lim_{n \to \infty}\frac{N_n}{n^{1-2\gamma}}=0\ \text{a.s.}
\end{equation}
Now, by using \eqref{Lim2} and \eqref{Lim1} in \eqref{Nnpdec}, we get the required result.
\end{proof}
\begin{theorem}
Let $p=C_{\beta,\gamma}$. Then,
\begin{equation*}
\limsup_{n \to \infty}\frac{|S_n|}{\sqrt{2n^{1-2\gamma}(\log n)^2\log\log n}}\le \frac{2(\beta+\gamma)}{\sqrt{1-2\gamma}}\ \text{a.s.}
\end{equation*}
\end{theorem}
\begin{proof}
From \eqref{cnasymcritt} and Lemma \ref{LemMnLIL}(i), we obtain
\begin{equation*}
\limsup_{n\to\infty}\frac{ac_nM_n}{\sqrt{2n^{1-2\gamma}(\log n)^2\log \log n}}=-\liminf_{n\to\infty}\frac{ac_nM_n}{\sqrt{2n^{1-2\gamma}(\log n)^2\log \log n}}=\frac{\beta+\gamma}{\sqrt{1-2\gamma}}\ \text{a.s.}
\end{equation*}
So,
\begin{equation}\label{limspMn}
\limsup_{n\to\infty}\frac{|ac_nM_n|}{\sqrt{2n^{1-2\gamma}(\log n)^2\log \log n}}=\frac{\beta+\gamma}{\sqrt{1-2\gamma}}\ \text{a.s.}
\end{equation}
Also, from Lemma \ref{lembnasym}(ii), Lemma \ref{unasymlem}(i) and Lemma \ref{lemNnLIL}, we get
\begin{equation}\label{limspNn}
\limsup_{n\to\infty}\frac{|N_n|}{\sqrt{2n^{1-2\gamma}(\log n)^2\log \log n}}\le \frac{\beta+\gamma}{\sqrt{1-2\gamma}}\ \text{a.s.}
\end{equation}
The required result follows on using \eqref{limspMn} and \eqref{limspNn} in \eqref{Nnpdec}.
\end{proof}
\subsubsection*{Case I(c)}
Here, $C_{\beta,\gamma}<p<C_{\beta}$ and $(\beta,\gamma)\ne (0,0)$.
\begin{theorem}
Let $C_{\beta,\gamma}<p<C_\beta$. Then, $\lim_{n\to\infty}S_n/n^{1-2\gamma}=0$ a.s.
\end{theorem} 
\begin{proof}
From \eqref{cnasymsupdiffu} and Lemma \ref{MnO}(i), we get
\begin{equation}\label{Mnasymiii}
\lim_{n \to \infty}\frac{c_nM_n}{n^{1-2\gamma}}=0\ \text{a.s.}
\end{equation}
Also, from Lemma \ref{unasymlem}(i) and Lemma \ref{lemNnLLN}(i), we have
\begin{equation}\label{Nnasymiii}
\lim_{n \to \infty}\frac{N_n}{n^{1-2\gamma}}=0\ \text{a.s.}
\end{equation}
Thus, by using \eqref{Mnasymiii} and \eqref{Nnasymiii} in \eqref{Nnpdec}, we get the required result.
\end{proof}
\begin{theorem}\label{thmdissulil}
For $C_{\beta,\gamma}<p<C_\beta$, we have
\begin{equation*}
\limsup_{n\to\infty}\frac{|S_n|}{\sqrt{2n^{1-2\gamma}\log\log n}}\le \frac{(\beta+\gamma)\sqrt{2\beta-2a(\beta+1)+1}+|a|(\beta+1)}{|a(\beta+1)-\beta-\gamma|\sqrt{2\beta-2a(\beta+1)+1}}\ \text{a.s.}
\end{equation*}
\end{theorem} 
\begin{proof}
From \eqref{cnasymsupdiffu} and Lemma \ref{LemMnLIL}(i), we get
\begin{equation}\label{Mndiffu}
\limsup_{n\to\infty} \frac{|c_nM_n|}{\sqrt{2n^{1-2\gamma}\log\log n}}=\frac{\beta+1}{|a(\beta+1)-\beta-\gamma|\sqrt{2\beta-2a(\beta+1)+1}}\ \text{a.s.}
\end{equation}
Also, from Lemma \ref{lembnasym}(iii), Lemma \ref{unasymlem}(i) and Lemma \ref{lemNnLIL}, we have
\begin{equation}\label{Nndiffu}
\limsup_{n\to\infty} \frac{|N_n|}{\sqrt{2n^{1-2\gamma}\log\log n}}\le \frac{\beta+\gamma}{|a(\beta+1)-\beta-\gamma|}\ \text{a.s.}
\end{equation}
By using \eqref{Mndiffu} and \eqref{Nndiffu} in \eqref{Nnpdec}, we get the required result.
\end{proof}
\begin{theorem}\label{thmasynrdiffu}
Let $C_{\beta,\gamma}<p<C_\beta$. Then, 
\begin{equation*}
\frac{S_n}{\sqrt{n^{1-2\gamma}}}\xrightarrow{d}\mathcal{N}(0,\eta^2),
\end{equation*}
where 
\begin{equation}\label{etasq}
\eta^2=\frac{(2\beta+1)^2+(1-2\gamma)(2\beta-2a(\beta+1)+1)}{(1-2\gamma)(2a(\beta+1)-2\beta-1)(2a(\beta+1)-2\beta-2+2\gamma)}.
\end{equation}
\end{theorem}
\begin{proof}
Let 
\begin{equation*}
A_n=\frac{1}{\sqrt{n^{1-2\gamma}}}\begin{pmatrix}
		ac_n & 0\\
		0 & 1
	\end{pmatrix}. 
\end{equation*}
Then, from \eqref{angMncal}, we have
{\small\begin{align}\label{asymnorQk}
A_n \langle \mathcal{M}\rangle_n A_n^t&=\frac{1}{n^{1-2\gamma}}\begin{pmatrix}
	a^2c_n^2s_n\ &\ ac_n\sum_{k=1}^{n}a_kb_k\mu_k\vspace{.2cm}\\
		ac_n\sum_{k=1}^{n}a_kb_k\mu_k\ &\ u_n
	\end{pmatrix}\nonumber\\
	&\ \ -\frac{(a(\beta+1))^2}{n^{1-2\gamma}}\begin{pmatrix}
		a^2c_n^2\sum_{k=1}^{n-1}\frac{a_{k+1}^2Y_k^2}{k^2}\ &\ ac_n\sum_{k=1}^{n-1}\frac{a_{k+1}b_{k+1}Y_k^2}{k^2\mu_{k+1}}\vspace{.2cm}\\
		ac_n\sum_{k=1}^{n-1}\frac{a_{k+1}b_{k+1}Y_k^2}{k^2\mu_{k+1}}\ &\ \sum_{k=1}^{n-1}\frac{b_{k+1}^2Y_k^2}{k^2\mu_{k+1}^2}
	\end{pmatrix},
\end{align}}
where $s_n=a_1^2\mu_1^2+a_2^2\mu_2^2+\dots+a_n^2\mu_n^2$ and $u_n=b_1^2+b_2^2+\dots+b_n^2$.
From \eqref{asyma_namn}, \eqref{asymmu_namn}, \eqref{eqsndiff}, \eqref{cnasymsupdiffu}, \eqref{bnsupdiffu} and Lemma \ref{unasymlem}(i),  we get
\begin{align}\label{alcas2}
\left.\begin{aligned}
a^2c_n^2s_n &\sim  \Big(\frac{a(\beta+1)}{a(\beta+1)-\beta-\gamma}\Big)^2\frac{n^{1-2\gamma}}{2\beta-2a(\beta+1)+1},\\
ac_n\sum_{k=1}^{n}a_kb_k\mu_k&\sim \frac{a(\beta+1)(\beta+\gamma)n^{1-2\gamma}}{(a(\beta+1)-\beta-\gamma)^2(a(\beta+1)-\beta+\gamma-1)},\\
u_n&\sim \Big(\frac{\beta+\gamma}{a(\beta+1)-\beta-\gamma}\Big)^2\frac{n^{1-2\gamma}}{1-2\gamma}.
	\end{aligned}\right\}
\end{align}
Also, from \eqref{asyma_namn}, \eqref{asymmu_namn},  \eqref{cnasymsupdiffu}, \eqref{bnsupdiffu} and Lemma \ref{MnO}(i), we have
\begin{align}\label{blcas2}
\left.\begin{aligned}
a^2c_n^2\sum_{k=1}^{n-1}\frac{a_{k+1}^2Y_k^2}{k^2}&=O(n^{-2\gamma}(\log n)^{1+\epsilon})\ \text{a.s.},\\
ac_n\sum_{k=1}^{n-1}\frac{a_{k+1}b_{k+1}Y_k^2}{k^2\mu_{k+1}}&=\begin{cases}
O((\log n)^{1+\epsilon})\ \text{a.s. if}\ \beta>0,\\
	O(n^{a-\gamma})\ \text{a.s. if}\ \beta=0,
\end{cases}\\
\sum_{k=1}^{n-1}\frac{b_{k+1}^2Y_k^2}{k^2\mu_{k+1}^2}&=\begin{cases}
	O(1)\ \text{a.s. if}\ \gamma>0,\\
	O((\log n)^{2+\epsilon})\ \text{a.s. if}\ \gamma=0.
\end{cases}
\end{aligned}
\right\}
\end{align}
Moreover, from \eqref{asymnorQk}, \eqref{alcas2} and \eqref{blcas2}, we get
\begin{equation*}
\lim_{n \to \infty}A_n \langle \mathcal{M}\rangle_nA_n^t=A=\begin{pmatrix}
		a_{11}\ &\ a_{12}\\
		a_{12}\ &\ a_{22}
	\end{pmatrix}\ \text{a.s.},
\end{equation*}
where 
\begin{align*}
a_{11}&= \Big(\frac{a(\beta+1)}{a(\beta+1)-\beta-\gamma}\Big)^2\frac{1}{2\beta-2a(\beta+1)+1},\\
a_{12}&= \frac{a(\beta+1)(\beta+\gamma)}{(a(\beta+1)-\beta-\gamma)^2(a(\beta+1)-\beta+\gamma-1)},\\
a_{22}&= \Big(\frac{\beta+\gamma}{a(\beta+1)-\beta-\gamma}\Big)^2\frac{1}{1-2\gamma}.
\end{align*}
Let $\|A_n\|=\sqrt{\lambda_{\max}(A_n^tA_n)}$. Then, $\|A_n\|\to 0$ as $n\to\infty$. Recall that $\delta_1=X_1$ and $|\delta_n|\le 2$, $n\ge 1$. So, for any $\epsilon >0$, from \eqref{deltavecMn}, we have
\begin{equation*}
\lim_{n \to \infty}\sum_{k=1}^{n}\mathbb{E}(\|A_n\Delta\mathcal{M}_k\|^2\mathbb{I}_{\|A_n\Delta\mathcal{M}_k\|>\epsilon}|\mathcal{F}_{k-1})=0\ \text{a.s.}
\end{equation*}
Thus, from Theorem A.1 of Bercu and Laulin (2021), we get
$A_n\mathcal{M}_n \xrightarrow{d} \mathcal{N}(0,A).$
Let $
v=\begin{pmatrix}
	1 & 1
\end{pmatrix}^t.$
Then, $v^tA_n\mathcal{M}_n \xrightarrow{d} \mathcal{N}(0,v^tAv)$. This completes the proof. 
\end{proof}
\begin{theorem}\label{meansqdig}
Let $C_{\beta,\gamma}<p<C_\beta$. Then,
$\mathbb{E}(S_n^2)\sim \eta^2n^{1-2\gamma}$,
where $\eta^2$ is as given in \eqref{etasq}.
\end{theorem}
\begin{proof}
From \eqref{Nnpdec}, we have
\begin{equation}\label{snsq}
\mathbb{E}(S_n^2)=\mathbb{E}(N_n^2)+2ac_n\mathbb{E}(M_nN_n)+(ac_n)^2\mathbb{E}(M_n^2).
\end{equation}
By using \eqref{asyma_namn}, \eqref{eqsndiff} and \eqref{EYnsq} in \eqref{Mnpredqu}, we get
\begin{equation}\label{expandmn}
\mathbb{E}(\langle M\rangle_n )\sim \Big(\frac{\Gamma(a(\beta+1)+1)}{\Gamma(\beta+1)}\Big)^2\frac{n^{2\beta-2a(\beta+1)+1}}{2\beta-2a(\beta+1)+1}.
\end{equation}
Also, by using \eqref{asyma_namn}, \eqref{asymmu_namn}, \eqref{EYnsq} and \eqref{bnsupdiffu} in \eqref{Nnpredqu} and \eqref{MnNnpredqu}, we obtain
\begin{equation}\label{expandNn}
\mathbb{E}(\langle N\rangle_n )\sim \Big(\frac{\beta+\gamma}{\beta+\gamma-a(\beta+1)}\Big)^2\frac{n^{1-2\gamma}}{1-2\gamma}
\end{equation}
and
\begin{equation}\label{expandmnNn}
\mathbb{E}(\langle M, N\rangle_n )\sim \frac{\Gamma(a(\beta+1)+1)}{\Gamma(\beta+1)} \Big(\frac{\beta+\gamma}{\beta+\gamma-a(\beta+1)}\Big)\frac{n^{\beta-\gamma-a(\beta+1)+1}}{\beta-\gamma-a(\beta+1)+1},
\end{equation}
respectively. Since $\{M_n^2-\langle M\rangle_n\}_{n\ge 1}$, $\{N_n^2-\langle N\rangle_n\}_{n\ge 1}$ and $\{M_nN_n-\langle M,N\rangle_n\}_{n\ge 1}$ are zero mean $\{\mathcal{F}_n\}_{n\ge1}$-martingales, the required result follows on using  \eqref{expandmn}, \eqref{expandNn} and \eqref{expandmnNn} in \eqref{snsq}.
\end{proof}
\begin{remark}
For $C_{\beta,\gamma}<p<C_\beta$,  the asymptotic behaviour of the walk is diffusive for $\gamma=0$ and subdiffusive for $\gamma>0$. This follows from Theorem \ref{meansqdig}.
\end{remark}
\subsubsection*{Case I(d)} Here, $C_{\beta,\gamma}<p=C_{\beta}$ and $(\beta,\gamma)\ne (0,0)$.
\begin{theorem}
For $p=C_\beta$, we have
\begin{equation*}
\lim_{n \to \infty}\frac{S_n}{(\log n)\sqrt{n^{1-2\gamma}}}=0\ \text{a.s.}
\end{equation*}
\end{theorem} 
\begin{proof}
From \eqref{cnasymsupdiffu} and Lemma \ref{MnO}(ii), we get
\begin{equation}\label{cnMncrit}
\lim_{n \to \infty}\frac{c_nM_n}{(\log n)\sqrt{n^{1-2\gamma}}}=0\ \text{a.s.}
\end{equation} 
Also, from Lemma \ref{unasymlem}(i) and Lemma \ref{lemNnLLN}(i), we obtain
\begin{equation}\label{Nncrit}
\lim_{n \to \infty}\frac{N_n}{(\log n)\sqrt{n^{1-2\gamma}}}=0\ \text{a.s.}
\end{equation}
Now, by using \eqref{cnMncrit} and \eqref{Nncrit} in \eqref{Nnpdec}, the required result follows.
\end{proof}
\begin{theorem}
Let $p=C_\beta$. Then,
{\small\begin{equation*}
\limsup_{n \to \infty}\frac{S_n}{\sqrt{2n^{1-2\gamma}(\log n)\log\log\log n}}=-\liminf_{n \to \infty}\frac{S_n}{\sqrt{2n^{1-2\gamma}(\log n)\log\log\log n}}=\frac{2\beta+1}{1-2\gamma}\ \text{a.s.}
\end{equation*}}
\end{theorem}
\begin{proof}
From Lemma \ref{LemMnLIL}(ii) and \eqref{cnasymsupdiffu}, we get
\begin{align}\label{crit1}
\limsup_{n \to \infty}\frac{ac_nM_n}{\sqrt{2n^{1-2\gamma}(\log n)\log\log\log n}}&=-\liminf_{n \to \infty}\frac{ac_nM_n}{\sqrt{2n^{1-2\gamma}(\log n)\log\log\log n}}\nonumber\\
&=\frac{2\beta+1}{1-2\gamma}\ \text{a.s.}
\end{align}
Also, by using \eqref{bnsupdiffu} and Lemma \ref{unasymlem}(i) in Lemma \ref{lemNnLIL}, we obtain
\begin{equation}\label{crit2}
\lim_{n\to\infty}\frac{N_n}{\sqrt{2n^{1-2\gamma}(\log n)\log\log\log n}}=0\ \text{a.s.}
\end{equation}
The required result follows on using 
\eqref{crit1} and \eqref{crit2} in \eqref{Nnpdec}.
\end{proof}
\begin{theorem}
Let $p=C_\beta$. Then,
\begin{equation*}
\frac{S_n}{\sqrt{n^{1-2\gamma}\log n}}\xrightarrow{d}\mathcal{N}\Big(0,\Big(\frac{2\beta+1}{1-2\gamma}\Big)^2\Big).
\end{equation*}
\end{theorem}
\begin{proof}
Let 
\begin{equation*}
B_n=\frac{1}{\sqrt{n^{1-2\gamma}\log n}}\begin{pmatrix}
		ac_n & 0\\
		0 & 1
	\end{pmatrix}.
\end{equation*}
 Then,
\begin{align*}
B_n\sum_{k=1}^{n}Q_k B_n^t\sim \begin{pmatrix}
		\frac{(2\beta+1)^2}{(1-2\gamma)^2} & 0\\
			0 & 0
	\end{pmatrix}\ \text{as $n\to\infty$},
\end{align*}
which follows from \eqref{asyma_namn}, \eqref{asymmu_namn}, \eqref{cnasymsupdiffu} and \eqref{bnsupdiffu}. Additionally, from \eqref{asyma_namn} and Lemma \ref{MnO}(ii), we get
\begin{equation*}
\lim_{n \to \infty}B_n\sum_{k=1}^{n-1}Q_{k+1} B_n^t=\begin{pmatrix}
			0 & 0\\
			0 & 0
		\end{pmatrix}\ \text{a.s.}
\end{equation*}
From this point the proof follows similar lines to that of Theorem \ref{thmasynrdiffu}. This completes the proof.
\end{proof}

\begin{theorem}\label{Snsq}
Let $p=C_\beta$. Then,
\begin{equation*}
\mathbb{E}(S_n^2)\sim \Big(\frac{2\beta+1}{1-2\gamma}\Big)^2n^{1-2\gamma}\log n.
\end{equation*}
\end{theorem}
\begin{proof}
By using \eqref{asyma_namn}, \eqref{cnasymsupdiffu} and \eqref{critYn} in \eqref{Mnpredqu}, we get
\begin{equation}\label{Nncritsql}
c_n^2\mathbb{E}(M_n^2)\sim \Big(\frac{2(\beta+1)}{1-2\gamma}\Big)^2n^{1-2\gamma}\log n
\end{equation} 
which follows from \eqref{sncrit}. Again, by using  \eqref{asyma_namn}, \eqref{asymmu_namn}, \eqref{bnsupdiffu}, \eqref{critYn} and Lemma \ref{unasymlem}(i) in \eqref{Nnpredqu}, we obtain
\begin{equation}\label{Nncritsq}
\mathbb{E}(N_n^2)\sim \Big(\frac{2(\beta+\gamma)}{1-2\gamma}\Big)^2\frac{n^{1-2\gamma}}{1-2\gamma}.
\end{equation}
Also, by using \eqref{asyma_namn}, \eqref{asymmu_namn}, \eqref{cnasymsupdiffu}, \eqref{bnsupdiffu} and \eqref{critYn} in \eqref{MnNnpredqu}, we get
\begin{equation}\label{crittMnNn}
c_n\mathbb{E}(M_nN_n)\sim \frac{-8(\beta+1)(\beta+\gamma)}{(1-2\gamma)^2} \frac{n^{1-2\gamma}}{1-2\gamma}.
\end{equation}
Finally, by using \eqref{Nncritsql}, \eqref{Nncritsq} and \eqref{crittMnNn} in \eqref{snsq}, we get the required result.
\end{proof}
\subsubsection*{Case I(e)} Here, $C_{\beta,\gamma}<C_\beta<p\le 1$ and $(\beta,\gamma)\ne (0,0)$.
\begin{theorem}
Let $C_\beta<p\le 1$. Then,
\begin{equation}\label{lim1}
\lim_{n \to \infty}\frac{S_n}{n^{a(\beta+1)-\beta-\gamma}}=L\ \text{a.s.},
\end{equation}
where $L$ is some non-degenerate random variable.
Moreover, we have the following mean square convergence:
\begin{equation}\label{Snsqq}
\lim_{n \to \infty}\mathbb{E}\Big(\Big(\frac{S_n}{n^{a(\beta+1)-\beta-\gamma}}-L\Big)^2\Big)=0.
\end{equation}
\end{theorem}
\begin{proof}
From \eqref{cnasymsupdiffu} and Lemma \ref{MnO}(iii), we get
\begin{equation}\label{sup1}
\lim_{n \to \infty}\frac{ac_nM_n}{n^{a(\beta+1)-\beta-\gamma}}=\frac{a\Gamma{(\beta+2)}M}{\Gamma{(a(\beta+1)+1})(a(\beta+1)-\beta-\gamma)}=L\ (\text{say})\ \text{a.s.}
\end{equation}
Also, from Lemma \ref{unasymlem}(i) and Lemma \ref{lemNnLLN}(i), we obtain
\begin{equation}\label{sup2}
\lim_{n \to \infty}\frac{N_n}{n^{a(\beta+1)-\beta-\gamma}}=0\ \text{a.s.}
\end{equation}
Thus, by using \eqref{sup1} and \eqref{sup2} in \eqref{Nnpdec}, we get \eqref{lim1}. 
Now, from the proof of Theorem 2.7 of Laulin (2022), we have
\begin{equation}\label{Mnsqq}
\lim_{n \to \infty}\mathbb{E}((M_n-M)^2)=0.
\end{equation}
Therefore, from \eqref{cnasymsupdiffu} and \eqref{Mnsqq}, we get
\begin{equation}\label{Mnsqq2}
\lim_{n \to \infty}\mathbb{E}\Big(\Big(\frac{ac_nM_n}{n^{a(\beta+1)-\beta-\gamma}}-L\Big)^2\Big)=0.
\end{equation}
Now, by using \eqref{asymmu_namn}, \eqref{EYnsq}, \eqref{bnsupdiffu} and Lemma \ref{unasymlem}(i) in \eqref{Nnpredqu}, we obtain
\begin{equation}\label{Nnsqq}
\lim_{n \to \infty}\mathbb{E}\Big(\Big(\frac{N_n}{n^{a(\beta+1)-\beta-\gamma}}\Big)^2\Big)=0.
\end{equation}
Finally, by using \eqref{Mnsqq2} and \eqref{Nnsqq} in \eqref{Nnpdec}, we get \eqref{Snsqq}.
\end{proof}
\begin{remark}
On substituting $\gamma=0$ in \eqref{Snsqq}, we get the mean square convergence of the amnesic ERW given in Eq. (2.12) of Laulin (2022).  
\end{remark}
\begin{remark}
From Remark 2.8 of Laulin (2022), the first and second order moments of $L$ are 
\begin{equation*}
\mathbb{E}(L)=\frac{a(2q-1)\Gamma(\beta+1)\Gamma{(\beta+2)}}{\Gamma(a(\beta+1)+1)(a(\beta+1)-\beta-\gamma)}
\end{equation*}
and 
\begin{equation*}
\mathbb{E}(L^2)=\frac{a^2\Gamma(2(a-1)(\beta+1)+1)(\Gamma(\beta+1))^2(\Gamma{(\beta+2))^2}}{(\Gamma((2a-1)(\beta+1)+1))^2(a(\beta+1)-\beta-\gamma)^2},
\end{equation*}
respectively.
\end{remark}
\begin{remark}
As $\mathbb{E}(|X_n-X|^r)\to 0$ implies $\mathbb{E}(|X_n|^r)\to\mathbb{E}(|X|^r)$ for all $r>0$, from \eqref{Snsqq}, we have
\begin{equation*}
\mathbb{E}(S_n^2)\sim \frac{a^2\Gamma(2(a-1)(\beta+1)+1)(\Gamma(\beta+1))^2(\Gamma{(\beta+2))^2}}{(\Gamma((2a-1)(\beta+1)+1))^2(a(\beta+1)-\beta-\gamma)^2}n^{2a(\beta+1)-2\beta-2\gamma}.
\end{equation*}
\end{remark}
\subsection{Case II} In this case $C_{\beta,\gamma}=C_\beta$.
We discuss the almost sure convergence, LIL, asymptotic normality and mean square displacement rate of the walk in different regimes.
\subsubsection*{Case II(a)} Here, $0\le p< C_{\beta}=C_{\beta,\gamma}$.
\begin{theorem}
Let $0\le p<C_\beta$. Then,\\
\noindent (i) for $p\ne 1/2$, we have
\begin{equation*}
\lim_{n\to\infty}\frac{S_n}{n^{2\beta-2a(\beta+1)+1}}=0\ \text{a.s.},
\end{equation*}
\noindent (i) for $p= 1/2$ and $\alpha>1/2$, we have
\begin{equation*}
\lim_{n\to\infty}\frac{S_n}{(\log n)^{\alpha}}=0\ \text{a.s.}
\end{equation*}
\end{theorem}
\begin{proof}
The proof follows on using \eqref{cnasymcritt}, Lemma \ref{MnO}(ii), Lemma \ref{unasymlem}(ii) and  Lemma \ref{lemNnLLN}(i) in \eqref{Nnpdec}. 
\end{proof}
The proof of next result follows similar lines as that of Theorem \ref{thmste}. So, it is omitted.
\begin{theorem}
Let $0\le p<C_\beta$. Then, \\
\noindent (i) for $p\ne 1/2$, we have
\begin{equation*}
\limsup_{n\to\infty}\frac{|S_n|}{\sqrt{2n^{2\beta-2a(\beta+1)+1}\log\log n}}\le \frac{2|ac_\infty \Gamma(a(\beta+1)+1)|}{\Gamma(\beta+1)\sqrt{2\beta-2a(\beta+1)+1}}\ \text{a.s.},
\end{equation*}
\noindent (ii) for $p= 1/2$, we have
\begin{equation*}
\limsup_{n\to\infty}\frac{S_n}{\sqrt{2(\log n)\log\log\log n}}=-\liminf_{n\to\infty}\frac{S_n}{\sqrt{2(\log n)\log\log\log n}}=1\  \text{a.s.}
\end{equation*}
\end{theorem}
\begin{theorem}
Let $p= 1/2$. Then, 
$S_n/\sqrt{\log n}\xrightarrow{d}\mathcal{N}(0,1)$.
\end{theorem}
\begin{proof}
By using Lemma \ref{unasymlem}(ii) in \eqref{Nnpredqu}, we have $
\lim_{n \to \infty}\langle N\rangle_n/\log n=1$ a.s.
Now, the proof follows similar lines to that of Theorem \ref{asymnn}. 
\end{proof}
\subsubsection*{Case II(b)} Here, $p=C_\beta=C_{\beta,\gamma}$.
\begin{theorem}
Let $p=C_\beta$. Then, for any $\alpha>3/2$, we have
$\lim_{n\to\infty}S_n/(\log n)^\alpha=0$ a.s.
\end{theorem}
\begin{proof}
For  $\alpha>3/2$, from \eqref{cnasymcritt} and Lemma \ref{MnO}(ii), we have
\begin{equation}\label{Mnlog}
\lim_{n \to \infty}\frac{c_nM_n}{(\log n)^\alpha}=0\ \text{a.s.},
\end{equation}
and from Lemma \ref{unasymlem}(ii) and Lemma \ref{lemNnLLN}(i),
 we have
\begin{equation}\label{Nnlog}
\lim_{n \to \infty}\frac{N_n}{(\log n)^\alpha}=0\ \text{a.s.}
\end{equation}
The required result follows on using \eqref{Nnlog} and \eqref{Mnlog} in \eqref{Nnpdec}.
\end{proof}
\begin{theorem}
Let $ p=C_\beta$. Then,
\begin{equation*}
\limsup_{n\to\infty}\frac{|S_n|}{\sqrt{2(\log n)^3\log\log\log n}}\le \frac{(\sqrt{3}+1)(2\beta+1)}{2\sqrt{3}}\ \text{a.s.}
\end{equation*}
\end{theorem}
\begin{proof}
From \eqref{cnasymcritt} and Lemma \ref{LemMnLIL}(ii), we have
\begin{equation}\label{Mnmo}
\limsup_{n\to\infty}\frac{|ac_nM_n|}{\sqrt{2(\log n)^3\log \log \log n}}=\frac{2\beta+1}{2}\ \text{a.s.}
\end{equation}
Also, from \eqref{bncrit}, Lemma \ref{unasymlem}(ii) and Lemma \ref{lemNnLIL}, we obtain
\begin{equation}\label{Nnmo}
\limsup_{n\to\infty}\frac{|N_n|}{\sqrt{2(\log n)^3\log \log \log n}}\le \frac{2\beta+1}{2\sqrt{3}}.
\end{equation}
Now, by using \eqref{Mnmo} and \eqref{Nnmo} in \eqref{Nnpdec}, we get the required result.
\end{proof}
\begin{theorem}
Let $p=C_\beta$. Then, 
\begin{equation*}
\frac{S_n}{(\log n)^{3/2}}\xrightarrow{d} \mathcal{N}\Big(0,\frac{(2\beta+1)^2}{12}\Big).
\end{equation*}
\end{theorem}
\begin{proof}
Let
\begin{equation*}
D_n=\frac{1}{(\log n)^{3/2}}
	\begin{pmatrix}
			ac_n & 0\\
			0 & 1
	\end{pmatrix}.
\end{equation*}
 By using \eqref{asyma_namn}, \eqref{asymmu_namn}, \eqref{cnasymcritt}, \eqref{bncrit},  Lemma \ref{MnO}(ii) and Lemma \ref{unasymlem}(ii) in \eqref{angMncal}, we obtain
	\begin{equation*}
\lim_{n \to \infty}D_n\sum_{k=1}^{n}Q_k D_n^t
=\frac{(2\beta+1)^2}{24}\begin{pmatrix}
			6 & -3\\
			-3 & 2
	\end{pmatrix}
\end{equation*}
and
\begin{equation*}
\lim_{n\to\infty}
D_n\sum_{k=1}^{n-1}Q_{k+1}D_n^t=\begin{pmatrix}
			0 & 0\\
			0 & 0
		\end{pmatrix}
		\  \text{a.s.}
\end{equation*}
Now, the proof follows similar lines to that of Theorem \ref{thmasynrdiffu}, we omit the details. 
\end{proof}
\begin{theorem}
Let $ p=C_\beta$. Then,
\begin{equation*}
\mathbb{E}(S_n^2)\sim \frac{(2\beta+1)^2}{12}(\log n)^3.
\end{equation*}
\end{theorem}
\begin{proof}
By using \eqref{asyma_namn}, \eqref{asymmu_namn}, \eqref{sncrit}, \eqref{cnasymcritt}, \eqref{bncrit} and \eqref{critYn} in \eqref{Mnpredqu}, \eqref{Nnpredqu} and \eqref{MnNnpredqu}, we obtain
\begin{align}\label{Sncrittsqq}
\left.\begin{aligned}
a^2c_n^2\mathbb{E}(M_n^2)&\sim \frac{(2\beta+1)^2}{4}(\log n)^3,\\
\mathbb{E}(N_n^2)&\sim \frac{(2\beta+1)^2}{12}(\log n)^3,\\
2ac_n\mathbb{E}(M_nN_n)&\sim  -\frac{(2\beta+1)^2}{4}(\log n)^3,
\end{aligned}
\right\}
\end{align}
respectively. Thus, the required result follows on using \eqref{Sncrittsqq} in \eqref{snsq}.
\end{proof}
\subsubsection*{Case II(c)} Here, $C_{\beta,\gamma}<C_\beta<p\le 1$.
\begin{theorem}
Let $C_\beta<p\le 1$. Then,
\begin{equation}\label{Snsubdifff}
\lim_{n \to \infty}\frac{S_n}{n^{a(\beta+1)-\beta-1/2}}=L\ \text{a.s.},
\end{equation}
where $L$ is as given in \eqref{sup1} with $\gamma=1/2$.
Also, we have the following mean square convergence:
\begin{equation}\label{Snsqsupdiff}
\lim_{n \to \infty}\mathbb{E}\Big(\Big(\frac{S_n}{n^{a(\beta+1)-\beta-1/2}}-L\Big)^2\Big)=0.
\end{equation}
\end{theorem}
\begin{proof}
From \eqref{cnasymsupdiffu} and Lemma \ref{MnO}(iii), we get
\begin{equation}\label{Mnsupdeff}
	\lim_{n \to \infty}\frac{ac_nM_n}{n^{a(\beta+1)-\beta-1/2}}=L\ \text{a.s.}
\end{equation}
Also, from Lemma \ref{unasymlem}(ii) and Lemma \ref{lemNnLLN}(i), we have
\begin{equation}\label{Nnsupdeff}
\lim_{n \to \infty}\frac{N_n}{n^{a(\beta+1)-\beta-1/2}}=0\ \text{a.s.}
\end{equation}
Thus, by using \eqref{Mnsupdeff} and \eqref{Nnsupdeff} in \eqref{Nnpdec}, we get \eqref{Snsubdifff}.

The proof of mean square convergence in \eqref{Snsqsupdiff} follows similar steps to that of the proof of \eqref{Snsqq}. This completes the proof. 
\end{proof}
\begin{remark}
From \eqref{Snsqsupdiff}, it follows that
\begin{equation*}
\mathbb{E}(S_n^2)\sim \frac{\Gamma(2(a-1)(\beta+1)+1)(2a\Gamma(\beta+1)\Gamma{(\beta+2))^2}}{(\Gamma((2a-1)(\beta+1)+1))^2(2a(\beta+1)-2\beta-1)^2}n^{2a(\beta+1)-2\beta-1}.
\end{equation*}
\end{remark}
\subsection{Case III} In this case $C_{\beta}<C_{\beta,\gamma}\le1$. We discuss almost sure convergence, LIL and mean square displacement rate for the walk $\{S_n\}_{n\ge1}$ in different regimes.

\subsubsection*{Case III(a)} Here, $0\le p<C_{\beta}<C_{\beta,\gamma}$.
\begin{theorem}\label{thmreff}
Let $0\le p<C_{\beta}$. Then,\\
\noindent (i) for $p\ne 1/2$, we have
\begin{equation*}
\lim_{n\to\infty}\frac{S_n}{n^{2\beta-2a(\beta+1)+1}}=0\ \text{a.s.},
\end{equation*}
\noindent (ii) for $p= 1/2$, we have
$\lim_{n\to\infty}S_n=S$ a.s. and $\lim_{n\to\infty}\mathbb{E}((S_n-S)^2)=0.$
Here, $S$ is some finite random variable.
\end{theorem}
\begin{proof}
By using Lemma \ref{MnO}(i), Lemma \ref{cnasymlem}(i), Lemma \ref{unasymlem}(iii) and Lemma \ref{lemNnLLN} in \eqref{Nnpdec}, we get the required results.
\end{proof}

\begin{remark}
For $p=1/2$, Theorem \ref{thmreff}(ii) implies that $\{S_n\}_{n\ge1}$ converges almost surely to a finite random variable. So, the walk is localized in this regime.
\end{remark}

\begin{theorem}
Let $0\le p<C_{\beta}$ and $p\ne 1/2$. Then,
\begin{equation*}
\limsup_{n\to\infty}\frac{|S_n|}{\sqrt{2n^{2\beta-2a(\beta+1)+1}\log\log n}}\le \frac{2|ac_\infty \Gamma(a(\beta+1)+1)|}{\Gamma(\beta+1)\sqrt{2\beta-2a(\beta+1)+1}}\ \text{a.s.}
\end{equation*}
\end{theorem}
\begin{proof}
From Lemma \ref{LemMnLIL}(i) and Lemma \ref{cnasymlem}(i), we have
\begin{align}
\limsup_{n\to\infty}\frac{ac_nM_n}{\sqrt{2n^{2\beta-2a(\beta+1)+1}\log \log n}}&=-\liminf_{n\to\infty}\frac{ac_nM_n}{\sqrt{2n^{2\beta-2a(\beta+1)+1}\log \log n}}\nonumber\\
&=\frac{|ac_\infty\Gamma(a(\beta+1)+1)|}{\Gamma(\beta+1)\sqrt{2\beta-2a(\beta+1)+1}}\ \text{a.s.}\label{reg3Mn}
\end{align}
Also, from Lemma \ref{lembnasym}(i), Lemma \ref{unasymlem}(iii) and Lemma \ref{lemNnLIL}, we get
\begin{equation}\label{reg3Nn1}
\limsup_{n\to\infty}\frac{|N_n|}{\sqrt{2n^{2\beta-2a(\beta+1)+1}\log\log n}}\le \frac{|ac_\infty\Gamma(a(\beta+1)+1)|}{\Gamma(\beta+1)\sqrt{2\beta-2a(\beta+1)+1}}\ \text{a.s.}
\end{equation}
The required result follows on using \eqref{reg3Mn} and \eqref{reg3Nn1} in \eqref{Nnpdec}.
\end{proof}
\subsubsection*{Case III(b)} Here, $p=C_{\beta}<C_{\beta,\gamma}$.

\begin{theorem}
Let $p=C_{\beta}$. Then,\\
\noindent (i) for $\alpha>1/2$, we have
\begin{equation}\label{assupdiff}
\lim_{n\to\infty}\frac{S_n}{(\log n)^\alpha}=0\ \text{a.s., and}
\end{equation}
\noindent (ii) the following LIL holds true:
\begin{equation}\label{lilsupdiff}
\limsup_{n\to\infty}\frac{|S_n|}{\sqrt{2(\log n)\log\log\log n}}\le \frac{|c_\infty|(2\beta+1) \Gamma(\beta+3/2)}{\Gamma(\beta+2)}\ \text{a.s.}
\end{equation}
\end{theorem}
\begin{proof} 
For  $\alpha>1/2$, from Lemma \ref{MnO}(ii) and Lemma \ref{cnasymlem}(i),  we get
\begin{equation}\label{Mnsh}
\lim_{n \to \infty}\frac{c_nM_n}{(\log n)^\alpha}\ \text{a.s.,}
\end{equation}
and from Lemma \ref{unasymlem}(iii) and Lemma \ref{lemNnLLN}(i),  we have
\begin{equation}\label{Nnsh}
\lim_{n \to \infty}\frac{N_n}{(\log n)^\alpha}\ \text{a.s.}
\end{equation}
By using \eqref{Mnsh} and \eqref{Nnsh} in \eqref{Nnpdec}, we get \eqref{assupdiff}.

Now, from Lemma \ref{cnasymlem}(i) and Lemma \ref{LemMnLIL}(ii), we obtain
\begin{equation}\label{Mnlilsh}
\limsup_{n \to \infty}\frac{|ac_nM_n|}{\sqrt{2(\log n)\log\log\log n}}= \frac{|c_\infty|(2\beta+1) \Gamma(\beta+3/2)}{2\Gamma(\beta+2)}\ \text{a.s.}
\end{equation}
Also, from \eqref{bndiffu}, Lemma \ref{unasymlem}(iii) and Lemma \ref{lemNnLIL}, we get
\begin{equation}\label{Nnlilsh}
\limsup_{n \to \infty}\frac{|N_n|}{\sqrt{2(\log n)\log\log\log n}}\le \frac{|c_\infty|(2\beta+1) \Gamma(\beta+3/2)}{2\Gamma(\beta+2)}\ \text{a.s.}
\end{equation}
Finally, by using \eqref{Mnlilsh} and \eqref{Nnlilsh} in \eqref{Nnpdec}, we obtain \eqref{lilsupdiff}.
\end{proof}
\subsubsection*{Case III(c)} Here, $C_\beta<p<C_{\beta,\gamma}$.
\begin{theorem}\label{thmref}
Let $C_\beta<p<C_{\beta,\gamma}$. Then, $\lim_{n\to\infty}S_n=\mathcal{M}$ a.s., where $\mathcal{M}$ is some finite random variable. Also, 
\begin{equation}\label{L2conv}
\lim_{n \to \infty}\mathbb{E}((S_n-\mathcal{M})^2)=0.
\end{equation}
\end{theorem}
\begin{proof}
From Lemma \ref{MnO}(iii) and Lemma \ref{cnasymlem}(i), we obtain
\begin{equation}\label{MnMnk}
\lim_{n\to\infty}ac_nM_n=ac_\infty M\ \text{a.s.} 
\end{equation}
Also, from \eqref{Nnpredqu}, we have $\langle N\rangle_n\le u_n$. Thus, by using Lemma \ref{unasymlem}(iii) in Lemma \ref{lemNnLLN}(ii), we get
\begin{equation}\label{NnNnk}
\lim_{n \to \infty}N_n=N\ \text{a.s.}
\end{equation}
Now, the required almost sure convergence result follows on using \eqref{MnMnk} and \eqref{NnNnk} in \eqref{Nnpdec}. 
 
As $\mathbb{E}(\langle N\rangle_n)\le u_n$, we have $\sup_{n\ge 1}\mathbb{E}(N_n^2)<\infty$, which follows from Lemma \ref{unasymlem}(iii). Thus, 
 $\{N_n\}_{n\ge1}$ is $\mathbb{L}^2$-bounded. Therefore,
 \begin{equation}\label{Nnkk}
 \lim_{n \to \infty}\mathbb{E}((N_n-N)^2)=0.
 \end{equation}
 Finally, by using \eqref{Mnsqq} and \eqref{Nnkk} in \eqref{Nnpdec}, we get \eqref{L2conv}. This completes the proof.
\end{proof}
\begin{remark}
In Theorem \ref{thmref}, as $\{S_n\}_{n\ge 1}$ converges almost surely to a finite random variable, the walk is localized in this regime.
\end{remark}	

For the subsequent regimes results, the proofs follow similar steps to that of Theorem \ref{thmref}. Hence, the details are omitted.

\subsubsection*{Case III(d)} Here, $C_\beta<p=C_{\beta, \gamma}<1$.
 
\begin{theorem}
Let $p=C_{\beta, \gamma}<1$. Then,
$\lim_{n \to \infty}S_n/\log n=L'$ a.s.,
where $L'=\gamma\Gamma(\beta+1)L_\beta$, where $L_\beta$ is given in Theorem 2.7 of Laulin (2022).
Moreover, we have
\begin{equation}\label{mmsqL2}
\lim_{n \to \infty}\mathbb{E}\Big(\Big(\frac{S_n}{\log n}-L'\Big)^2\Big)=0.
\end{equation}
\end{theorem}
\begin{remark}
From \eqref{mmsqL2}, we have
$\mathbb{E}(S_n^2)\sim (\gamma\Gamma(\beta+1))^2\mathbb{E}(L_\beta^2)(\log n)^2
$,
where $\mathbb{E}(L_\beta^2)$ is given in Eq. (2.14) of Laulin (2022).
\end{remark}
\subsubsection*{Case III(e)} Here, $C_\beta<C_{\beta,\gamma}<p\le 1$.
\begin{theorem}
Let $C_{\beta,\gamma}<p\le 1$. Then, 
$
\lim_{n \to \infty}S_n/n^{a(\beta+1)-\beta-\gamma}=L$ a.s.,
where $L$ is as given in \eqref{sup1}.
Also, we have
\begin{equation*}
\lim_{n \to \infty}\mathbb{E}\Big(\Big(\frac{S_n}{n^{a(\beta+1)-\beta-\gamma}}-L\Big)^2\Big)=0.
\end{equation*}
\end{theorem}
\subsubsection*{Case III(f)} Here, $C_\beta<p=C_{\beta,\gamma}=1$.
\begin{theorem}
Let $p=C_{\beta,\gamma}=1$. Then,
$\lim_{n \to \infty}S_n/\log n=M$ a.s.,
where $M$ is as given in Lemma \ref{MnO}(iii).
Also, we have
\begin{equation*}
\lim_{n \to \infty}\mathbb{E}\Big(\Big(\frac{S_n}{\log n}-M\Big)^2\Big)=0.
\end{equation*}
\end{theorem}

\section*{Acknowledgement}
The second author thanks Government of India for the grant of Prime Minister's Research Fellowship, ID 1003066.

\end{document}